\documentclass[a4paper,12pt,reqno]{amsart}
\usepackage{amssymb,amsfonts,amsmath,amscd,amsthm,latexsym,euscript}
\usepackage{cite}
\usepackage{times}

\usepackage[bookmarks=false]{hyperref}
\hypersetup{%
    colorlinks=true,        
    linkcolor=blue,          
    citecolor=blue,         
    urlcolor=blue           
    }

\usepackage{url}

\newtheorem{theor}{Theorem}

\theoremstyle{definition}

\newtheorem{define}{Definition}

\newtheorem{condition}{Condition}

\newtheorem{ex}{Example}
\newtheorem{counterexample}[ex]{Counterexample}
\theoremstyle{remark}
\newtheorem{rem}{Remark}

\theoremstyle{definition}
\theoremstyle{definition}

\newcommand{\BBR}{\mathbb{R}}
\newcommand{\BBN}{\mathbb{N}}
\newcommand{\BBS}{\mathbb{S}}

\newcommand{\cX}{{\EuScript X}}    

\newcommand{\bx}{{\boldsymbol{x}}}

\DeclareMathOperator{\Diffeo}{Diffeo}

\DeclareMathOperator{\Hom}{Hom}

\newcommand{\subout}{{\text{\textup{out}}}}
\newcommand{\subin}{{\text{\textup{in}}}}

\newcommand{\by}[1]{\textrm{{#1}}}
\newcommand{\jour}[1]{\textit{{#1}}}
\newcommand{\vol}[1]{\textbf{{#1}}}
\newcommand{\book}[1]{\textit{{#1}}}

\begin{document}

\title[Jacobi identities for Wronskians in multidimension]{Jacobi identities for Wronskian determinants\\[3pt] over multidimension}


\author[A.\,V.\,Kiselev]{Arthemy V.\ Kiselev}
\thanks{\textit{Address}:\quad 
Bernoulli Institute for Mathematics, Computer Science \&\ 
Artificial Intelligence, 
University of Groningen, P.O.\,Box\:407, 9700\,AK Groningen, 
The Netherlands.}

\dedicatory{\textup{Based on the talk given at the XIII International symposium on Quantum Theory and Symmetries -- QTS13 (Yerevan, Armenia, 28 July -- 1 August 2025).
}}



\date{5 November 2025}

\begin{abstract}
The generalised Wronskian of differential order $k\geqslant 1$ for $N$ functions $f_1$,\ $\ldots$,\ $f_N$ in $d\geqslant 1$ independent variables $x^1$,\ $\ldots$,\ $x^d$ is the determinant of the matrix with these functions' derivatives 
$\partial^{|\sigma_i|} f_j / \partial (x^1)^{\sigma_i^1}\cdots \partial (x^d)^{\sigma_i^d}$ (of orders $0 \leqslant |\sigma_i| \leqslant k$), where the multi\/-\/indices $\sigma_i$ mark (all or part of) fibre variables $u_{\sigma_i}$ in the $k$th jet space $J^k\bigl(\BBR^d\to\BBR\bigr)$.
We prove that these (in)complete Wronskians --\,provided that their lowest\/-\/order parts are complete at differential orders $\ell\leqslant 1$\,-- over the $d$-\/dimensional base satisfy the table of bi\/-\/linear, Jacobi\/-\/type identities for Schlessinger\/--\/Stasheff's strongly homotopy Lie algebras. 
\end{abstract}




\maketitle

\section{Introduction}\label{SecIntroduction}
\noindent%
The Wronskian determinants are used to inspect linear (in)\/dependence 
of functions\footnote{\label{FootWronskBehaves}
The Wronskian of $N$ scalar functions has conformal weight $N(N-1)/2$, so 
itself is not a scalar function if $N>1$. Likewise, the arguments $f_j$ can be not scalar functions (of conformal weight $0$) but coefficients of positive\/-\/order differential operators on~$\BBR$, hence themselves behave under coordinate changes on the base $\BBR\ni x$.}
$f_1$,\ $\ldots$,\ $f_N$ (differentiable enough many times on an interval in~$\BBR$): if their Wronskian,
\[
W^{0,1,\ldots,N-1}(f_1,\ldots,f_N)=\det\bigl(\partial^{i-1} f_j / \partial x^{i-1}
\bigr)
\]
is not identically zero, 
then they are linearly independent.

\begin{ex}
\label{FootPeanoCounterex}
$W^{0,1}(x,x^2) = \left| \begin{smallmatrix} x & x^2 \\ 1 & 2x 
\end{smallmatrix} \right| = x^2 \not\equiv 0$ on $[-1,1]\ni x$.
Still the vanishing of the Wronskian on an interval does not yet imply linear independence; here is Peano's counterexample:
$W^{0,1}( x^2, x|x| ) \equiv 0$ on $[-1,1]\ni x$, but the functions $x^2$ and $x|x|$ are linearly independent on any open neighbourhood of the origin.
\end{ex}

For differentiable functions $f_j\in C^k\bigl( \BBR^{d\geqslant 1} \to \BBR\bigr)$ in many variables $x^1$,\ $\ldots$,\ $x^d$, the concept of Wronskian was re\/-\/discovered over decades by many authors from various disciplines
(see \cite{LeVeque1956,Schmidt1980,Wolsson1989a}, also~\cite{ForKac} referring to 2002--3 or A.\,G.\,Khovanskii in 2003--4 
(private communication)).

To generalise the Wronskian determinants to spaces of functions on $\BBR^d$ of dimensions $d\geqslant 1$, fix the differential order $k\geqslant 1$ (and work with arguments $f_j\in C^k(\BBR^d\to\BBR)$). 
List all the (multi\/-\/indices of) derivatives\footnote{\label{FootDimJetFibre}
\textbf{Lemma.} The dimension of $k$th jet fibre in the jet bundle $J^k(\BBR^d\to\BBR)$, counting $\sigma=\varnothing$ as well, equals $\binom{d+k}{d}$ under the natural assumption that $u_{xy}=u_{yx}$, etc., for all~$u_\sigma$.}
$u_{\sigma_i}$ of intermediate orders: $0\leqslant |\sigma_i|\leqslant k$.
We say that in a fixed system of coordinates on the affine 
$\BBR^d$, the \emph{complete} $k$th differential order Wronskian $W^{d\geqslant 1}_{k\leqslant 1}$ of $N=\binom{d+k}{d}$ arguments $f_j$ viewed as functions from $C^k(\BBR^d\to\BBR)$ is defined by the formula
\begin{equation}\label{EqFullWronskD}
W^{d\geqslant 1}_{k\geqslant 1} \bigl(f_1,\ldots,f_N\bigr) =
\det\Bigl(\partial^{|\sigma_i|} f_j \bigl/ \partial (x^1)^{\sigma_i^1}\ldots\partial (x^d)^{\sigma_i^d} \Bigr),
\end{equation}
where $\sigma_i=(\sigma_i^1,\ldots,\sigma_i^d)=(\#x^1,\ldots,\#x^d)$ runs over the multi\/-\/indices of $k$th jet's fibre coordinates $u_{\sigma_i}$; the index $i$ enumerates rows and $j$ counts columns ($1\leqslant i,j\leqslant N$).

\begin{ex}[{cf.~\cite{ForKac}}]\label{ExTernaryBr}
Over the Cartesian plane $\BBR^{d=2}\ni(x,y)$ and for the order bound $k=1$, the complete Wronskian is $W^{d=2}_{k=1} = \mathbf{1}\wedge \partial/\partial x\wedge \partial/\partial y$, that is
\begin{equation}\label{EqTernaryBracket}
W^{d=2}_{k=1}\bigl(f,g,h\bigr) = \begin{vmatrix} f & g & h \\ f_x & g_x & h_x \\ f_y & g_y & h_y \end{vmatrix},\qquad f,g,h\in C^1(\BBR^2\to\BBR).
\end{equation}
This ternary operator is tri\/-\/linear 
and totally antisymmetric w.r.t.\ its arguments: 
\[
W^{d=2}_{k=1}\bigl(\pi(f),\pi(g),\pi(h)\bigr)
= (-)^\pi W^{d=2}_{k=1}\bigl(f,g,h\bigr)
\]
for any permutation $\pi\in\mathbb{S}_3$.
\end{ex}

\begin{rem}\label{FootPeano2D}
The (in)complete generalised Wronskians over dimensions $d\geqslant 1$, which we presently describe, are subject to the same reservations --\,about their (in)sufficience to show the linear (in)dependence of functions\,-- as in the classical case of $d=1$. For \textmd{example}, the three functions $f(x,y)=x^2 y^2$,\ $g(x,y)=x|x|\cdot y^2$,\ and $h(x,y)=x^2\cdot y|y|$ are linearly independent on the square $[-1,1]\times[-1,1]\ni(x,y)$, yet their complete first\/-\/order generalised Wronskian from Eq.~\eqref{EqTernaryBracket}, see above, 
vanishes identically on the entire domain of definition: 
$W^{d=2}_{k=1} \bigl(x^2y^2$,$x|x|\cdot y^2$,$x^2\cdot y|y|\bigr)\equiv 0$.
Indeed, for $x\geqslant0$ (and any $y\in\BBR$) determinant's 1st and 2nd columns coincide; for $y\geqslant 0$ (and any $x\in\BBR$) the 1st and 3rd columns coincide, whereas on $x<0$ and $y<0$, the 2nd and 3rd columns equal minus the first.
\end{rem}

\begin{define}
The generalised Wronskian determinant $'W^{d\geqslant 1}_{\underline{k\geqslant 1}}$ is \emph{incomplete} if its list $\{\sigma_i \}$ lacks certain multi\/-\/indices; exclusion is allowed only for the \emph{highest}\/-\/order derivatives (with $|\sigma_i|=k$).
\end{define}

\begin{ex}\label{FootExIncompleteWdimD}
For dimension $d=2$ and order $k=2$, by excluding the last multi\/-\/index $\sigma_6 = yy=(0,2)$ of top order $|\sigma_6|=2$ from their full list $\{\varnothing$,$x$,$y$,$xx$,$xy$,$yy\}$ at all orders $0\leqslant\ell\leqslant k=2$, we obtain the incomplete Wronskian dererminant of size $5\times 5$. Clearly, if this determinant already is not identically zero for five given functions, they are linearly independent (irrespective of the values of their second partial derivatives in~$y$).
\end{ex}

\subsection*{Preliminaries: strongly homotopy Lie algebras}
Let us recall that the usual Wronskians (over dimension $d=1$, see \cite{Dzhuma2002})
and complete generalised Wronskians (over $d>1$ and of differential orders $k,\ell\geqslant 1$, see \cite{ForKac}) satisfy the two\/-\/parametric (as $k,\ell\in\BBN$) table of Jacobi\/-\/type identities, bilinear w.r.t.\ the $N$-ary structures of orders $k$ and $\ell$, for strongly homotopy Lie algebras with zero differential.\footnote{\label{FootRefLadaStasheffSH}
The reader is addressed to the notes \cite{LadaStasheff1993} for definitions and 
physical context: how homotopy Lie algebras appear in various models, see literature references therein.} 
Namely, denote by $A\mathrel{{:}{=}} C^{r\gg 1}(\BBR^d\to\BBR)$ the algebra of `good' functions; let $\Delta\in\Hom_\Bbbk\bigl(\bigwedge^{N_\subout} A$,$A\bigr)$ and $\nabla\in\Hom_\Bbbk\bigl(\bigwedge^{N_\subin} A$,$A\bigr)$ be $\Bbbk$-\/linear totally antisymmetric operators on~$A$. By definition, the \emph{action} of $\Delta$ on $\nabla$ is 
$\Delta[\nabla]\bigl(a_1,\ldots,a_{N_\subout+N_\subin -1}\bigr) =
  \bigl[ N_\subin! (N_\subout -1)! \bigr]^{-1} \cdot{} $
\[
\sum_{\tau\in\BBS_{N_\subin+N_\subout-1}}
(-)^\tau \Delta\bigl(\nabla\bigl( a_{\tau(1)},\ldots,a_{\tau(N_\subin)} \bigr),\\
a_{\tau(N_\subin+1)},\ldots,a_{\tau(N_\subin+N_\subout-1)} \bigr);
\]
the 
$(N_\subin+N_\subout-1)$-ary operator $\Delta[\nabla]$
is totally antisymmetric in~
$a_m\in A$.

Dzhumadil'daev proved in~\cite{Dzhuma2002} for $d=1$ (cf.\ \cite{ForKac} with any $d\geqslant 1$) that 
Wronskian determinants of arbitrary orders $k_\subout,\ell_\subin
$ satisfy the table of Jacobi 
identities,
\begin{equation}\label{EqJacobiTableFullWWdim1}
W^{d=1}_{k_\subout\geqslant 1} \bigl[ W^{d=1}_{\ell_\subin\geqslant 1} \bigr] = 0.
\end{equation}
In the recent work~\cite{PRG25Wronsk} we recall in which way the Jacobi identities of this specific type for $N$-ary structures, given on~$A$ by the Wronskians, appear in the course of homotopy deformation of the Lie algebra $\cX^1(\BBR^{d=1})$ of vector fields on a one\/-\/dimensional base manifold.

\section{Jacobi identities for (in)complete Wronskians}\label{SecGrowWronsk}
\noindent%
We now strengthen the result in~\cite{ForKac}, extending the table of Jacobi\/-\/type identities \eqref{EqJacobiTableFullWWdim1} (over $d=1$ and over $d>1$ for complete sets of top\/-\/order derivatives in either Wronskian) to the case of \emph{incomplete} Wronskians: they may lack subsets of derivatives in the highest orders $k,\ell>1$
over dimension~$d>1$.

\begin{condition}\label{CondOrd1Full}
In what follows (and in contrast with Counterexample~\ref{CounterExNonZeroJac} on p.~\pageref{CounterExNonZeroJac} below), the (in)complete Wronskians
$'W^{d\geqslant 1}_{\underline{s\geqslant 1}}$ are admissible only if their set of \emph{first}\/-\/order derivatives is complete,
i.e.\ every $\partial/\partial x^a$ shows up in $'W^{d\geqslant 1}_{\underline{s\geqslant 1}} = \mathbf{1}\wedge\partial_{x^1}\wedge\ldots\wedge\partial_{x^d}\wedge\ldots$; omission of multi\/-\/indice(s) can occur only in the highest order~$s>1$.
\end{condition}

\begin{ex}\label{FootExWadmissible}
Consider again the example ($d=2$,$k=2$) in footnote~\ref{FootExIncompleteWdimD} on p.~\pageref{FootExIncompleteWdimD}: admissible are, for instance, the incomplete Wronskians $\mathbf{1}\wedge\partial_x\wedge\partial_y\wedge\partial_{xx}$ or
$\mathbf{1}\wedge\partial_x\wedge\partial_y\wedge\partial_{xx}\wedge\partial_{yy}$, etc., but not $\mathbf{1}\wedge\partial_y\wedge\partial_{xx}\wedge\partial_{xy}\wedge\partial_{yy}$ which lacks $\partial_x=\partial/\partial x$ in order~$1$.
\end{ex}

We recall from Lemma in footnote~\ref{FootDimJetFibre} on p.~\pageref{FootDimJetFibre} that $N=\binom{d+k}{d}$ is the number of different (modulo $u_{xy}=u_{yx}$, etc.) partial derivatives of all orders $0$,$\ldots$,$k\geqslant 1$ w.r.t.\ the $d\geqslant1$ independent variables $x^1$,$\ldots$,$x^d$. 
The \emph{complete} generalised Wronskians
$W^{d\geqslant1}_{k\geqslant1} = \mathbf{1}\wedge\partial_{x^1}\wedge\ldots\wedge\partial_{x^d}\wedge\ldots\wedge\partial_{x^d}^k$ contain all the multi\/-\/indices of these derivatives, starting from $\varnothing$ in the leading wedge factor $\mathbf{1}$ till all of the derivations $\partial^{|\sigma|}/\partial\bx^\sigma$ with $|\sigma|=k$. In this case of complete sets, we proved
that Jacobi identities \eqref{EqJacobiTableFullWWdim1} extend from $d=1$ to arbitrary dimensions~$d\geqslant1$.

\begin{ex}[cf.~\cite{ForKac}]\label{FootExTernaryJacobiDim2}
Over $d=2$ at order $k=1$, 
ternary bracket~\eqref{EqTernaryBracket} satisfies the ternary Jacobi identity,
$\mathbf{1}\wedge\partial_x\wedge\partial_y \bigl[ \mathbf{1}\wedge\partial_x\wedge\partial_y \bigr] = 0$, which is verified by direct calculation.
\end{ex}

\begin{theor}[\cite{ForKac}]\label{ThJacobiTableFullWWdimD}
Over $d\geqslant 1$, the complete generalised
Wronskians satisfy the Jacobi identities 
$W^{d\geqslant 1}_{k_\subout\geqslant 1} \bigl[ W^{d\geqslant 1}_{\ell_\subin\geqslant 1} \bigr] = 0$ for all differential orders $k_\subout,\ell_\subin\in\BBN$.
\end{theor}

\begin{proof}[Proof scheme \textup{(cf.~\cite[Prop.\,5]{PRG25Wronsk} for $d=1$ and~\cite{ForKac} for $d\geqslant1$)}]
By construction, the ope\-ra\-tor $W^d_{k}\bigl[ W^d_{\ell} \bigr]$ is totally antisymmetric w.r.t.\ its $N_\subin + N_\subout -1$ arguments; hence, to be nonzero, this Jacobiator must act by pairwise non\/-\/coinciding differentiations $\partial^{|\sigma|} / \partial\bx^\sigma$ on all of its arguments. The key idea is to estimate the overall sum of their differential orders (in other words, count all $\partial/\partial x^a$ at hand for whatever $a\in\{1$,$\ldots$,$d\}$).
Without loss of generality suppose $\ell_\subin = \max(k_\subout$,$\ell_\subin)\geqslant k_\subout$; the complete Wronskian $W^{d\geqslant 1}_{\ell\geqslant 1} = \mathbf{1}\wedge \partial/\partial{x^1}\wedge\ldots\wedge\partial^{\ell}/\partial(x^d)^{\ell}$ contains all the differentiations of all orders $0$,$\ldots$,$\ell_\subin$.
To have more derivatives 
that would not repeat the previously considered ones, higher\/-\/order operators (of orders${}>\ell_\subin$) are needed for the second, \ldots, last arguments of the other Wronskian. 
The other Wronskian, 
$\mathbf{1}\wedge\langle{}$terms of order${}\leqslant\ell_\subin\rangle$, contains the right number of positive\/-\/order terms, but each of those differential orders does not exceed $\ell_\subin < \ell_\subin+1$, contrary to the required.
Therefore, at least one differentiation repeats twice, and the antisymmetrisation cancels out the entire operator's action.
\end{proof}

\begin{rem}\label{FootRhoNearW}
From the proof it is readily seen that Wronskians can be pre\/-\/multiplied by an arbitrary factor $\varrho(\bx)$ --\,which, via the Leibniz rule, can absorb part of the derivatives acting on the inner Wronskian when it becomes the first argument of the outer structure,\,-- still preserving the statement of 
Theorem~\ref{ThJacobiTableFullWWdimD}: all the Jacobiators vanish.
\end{rem}

The flaw of assertion self in Theorem~\ref{ThJacobiTableFullWWdimD} is that over big dimension $d>1$, the matrix size of either Wronskian determinant leaps from
$N(d,r)=\binom{d+r}{d}$ to $N(d,r+1)=\binom{d+r+1}{d} > N(d,r)+1$ as the order $r$ increments by~$+1$. We claim that whenever $k,\ell>1$, the conclusion (with basically the same proof) can be strengthened: having completed the Wronskians $W^{d\geqslant1}_{k_\subout\geqslant1}$ and $W^{d\geqslant1}_{\ell_\subin\geqslant 1}$ at the preceding differential orders, we can \emph{gradually} accumulate either Wronskian in the next order $k_\subout+1$ and $\ell_\subin+1$ by incorporating new derivatives one by one. 
Along many intermediate scenarios (for choosing the subsets of multi\/-\/indices in the next, not yet complete differential order), the \emph{complete} Wronskians
$W^{d\geqslant1}_{k_\subout+1 > 1}$ and $W^{d\geqslant1}_{\ell_\subin+1 > 1}$ are attained.

In what follows we assume again that, without loss of generality, $\ell_\subin \geqslant k_\subout$ (otherwise, swap `in'${}\rightleftarrows{}$`out'). We stress that under Condition~\ref{CondOrd1Full}, both the (in)complete Wronskians
$'W^{d\geqslant1}_{\underline{k_\subout\geqslant1}}$ and $'W^{d\geqslant1}_{\underline{\ell_\subin\geqslant 1}}$ must contain the complete sets of \emph{first}\/-\/order derivations $\partial 
_{x^1}\wedge\ldots\wedge\partial_{x^d}$.

\begin{theor}\label{ThJacobiTableFullInPartOut}
Suppose that in the senior order $\ell_\subin\geqslant k_\subout\geqslant 1$, the Wronskian $W^{d\geqslant 1}_{\ell_\subin\geqslant 1}$ is complete;
the other Wronskian $'W^{d\geqslant 1}_{\underline{k_\subout\geqslant 1}}$ can be either incomplete in its highest order $1 < k_\subout \leqslant \ell_\subin$ or complete of order $k_\subout=1$, $W^{d\geqslant 1}_{k_\subout=1}$ without any derivatives of order${}\geqslant2$.
Then the Jacobi identity holds: $'W^{d\geqslant 1}_{\underline{k_\subout\geqslant 1}} \bigl[ W^{d\geqslant 1}_{\ell_\subin\geqslant 1} \bigr] = 0$.
\end{theor}

%
\begin{proof}
Here, the proof repeats --\,word by word
\,-- that 
of Theorem~\ref{ThJacobiTableFullWWdimD}.
\end{proof}

\begin{theor}\label{ThJacobiTablePartInEnoughPartOut}
Suppose that in its senior order $\ell_\subin > 1$ the Wronskian $'W^{d\geqslant 1}_{\underline{\ell_\subin>1}}$ is incomplete, still the (in)complete other Wronskian determinant $'W^{d\geqslant 1}_{\underline{k_\subout\geqslant 1}}$ of size $N_\subout\times N_\subout$ with $k_\subout \leqslant \ell_\subin$ is such that $N_\subout - 1 > \bigl($the number of highest, $\ell_\subin$th\/-\/order derivatives missing in the top of the incomplete senior\/-\/order
Wronskian $'W^{d\geqslant 1}_{\underline{\ell_\subin > 1}}\bigr)$.
Then the Jacobi identity holds: $'W^{d\geqslant 1}_{\underline{k_\subout\geqslant 1}} \bigl[ 'W^{d\geqslant 1}_{\underline{\ell_\subin > 1}} \bigr] = 0$.
\end{theor}

%
\begin{proof}
We only need to bound the (sum of) orders $|\sigma|$ of $\partial^{|\sigma|}/\partial \bx^\sigma$. Do `complete' the senior\/-\/order Wronskian by fictitiously moving 
the lacking number of derivatives from the lower\/-\/order Wronskian --- 
neglecting any repetitions of multi\/-\/indices $\sigma$ and pretending that all 
the carried derivatives are senior order, $\ell_\subin$.
One Wronskian now completed, the other again does not attain the required dif\-fe\-ren\-ti\-al order $\ell_\subin +1$ in each of its second, $\ldots$, last remaining wedge factor.
\end{proof}

\begin{theor}\label{ThJacobiTablePartInPartOutInsuff}
Suppose that in its senior order $\ell_\subin > 1$ the Wronskian $'W^{d\geqslant 1}_{\underline{\ell_\subin > 1}}$ is incomplete, and the (in)complete other Wronskian determinant $'W^{d\geqslant 1}_{\underline{k_\subout\geqslant 1}}$ of size $N_\subout\times N_\subout$ with $k_\subout \leqslant \ell_\subin$ is such that
$N_\subout - 1 \leqslant \bigl($the number of highest, $\ell_\subin$th\/-\/order derivatives missing in the top of the incomplete senior\/-\/order
Wronskian $'W^{d\geqslant 1}_{\underline{\ell_\subin > 1}}\bigr)$.
Then the Jacobi identity holds: $'W^{d\geqslant 1}_{\underline{k_\subout\geqslant 1}} \bigl[ 'W^{d\geqslant 1}_{\underline{\ell_\subin > 1}} \bigr] = 0$.
\end{theor}

%
\begin{proof}
Indeed, the outer Wronskian $'W^{d\geqslant 1}_{\underline{k_\subout\geqslant 1}}$ cannot supply $N_\subout - 1$ derivatives of order $\ell_\subin > 1$ --\,to let the Jacobiator act by non\/-\/coinciding derivations on all of its arguments\,-- because
$'W^{d\geqslant 1}_{\underline{k_\subout\geqslant 1}}$ contains at least one lowest\/-\/order derivation $\partial/\partial x^a$, yet their full set is already present in
$'W^{d\geqslant 1}_{\underline{\ell_\subin > 1}}$ of higher order.
\end{proof}

Only the last case, Theorem~\ref{ThJacobiTablePartInPartOutInsuff} explicitly relies on the assumption $\ell_\subin > 1$ and Condition~\ref{CondOrd1Full} that the set of first\/-\/order multi\/-\/indices is full in $'W^{d\geqslant 1}_{\underline{\ell_\subin > 1}}$. In fact, the outer Wronskian can then be incomplete of order~$1$\,!

\begin{counterexample}\label{CounterExNonZeroJac}
But this is what happens when the above assumption ($\ell_\subin > 1$) and Condition~\ref{CondOrd1Full} are ignored: over $d=2$, we have
\[
\mathbf{1}\wedge \partial/\partial y \bigl[ \mathbf{1}\wedge \partial/\partial x \bigr] = 2\cdot W^{d=2}_{r=1} \not\equiv 0,
\] 
i.e.\ the action of one incomplete low\/-\/order Wronskian on the other of same type recovers ternary bracket~\eqref{EqTernaryBracket}.
\end{counterexample}

\section{Conclusion}\label{SecConclDeformDary?}
\noindent%
We established the `no\/-\/gaps' set of Jacobi identities (from the entire table of identities for the strongly homotopy Lie algebra with zero differential): we are now free to increment the size of either Wronskian determinant by $+1$, that is without huge leaps to the dimension of the next, higher\/-\/order jet fibre. Through Condition~\ref{CondOrd1Full} (contrasted by Counterexample~\ref{CounterExNonZeroJac}), the Wronskians over multidimension $d>1$
--\,participating in the homotopy of \emph{unknown} Lie\/-\/algebraic 
object\,-- still reproduce the fact that over $d=1$, the original structure to\/-\/deform was the Lie algebra $\cX^1(\BBR^1)$ of vector fields on the line, whence the wedge $\mathbf{1}\wedge \partial/\partial x$ (from the commutator of vector fields on $\mathbb{R}^1$) was seen in every Wronskian $W^{0,1,\ldots,N-1} = \mathbf{1}\wedge \partial_x\wedge\ldots\wedge\partial_x^{N-1}$. The deformation of $\cX^1(\BBR^1)$ ran over higher\/-\/order differential operators $w_j(x)\,\smash{\partial_x^{N/2}}$ for even $N=2p\in\BBN$ and over still\/-\/unrecognised objects for $N$ odd. The nature of deformation's higher\/-\/order terms over $d>1$ is not yet identified. The contrast of three new Theorems~\ref{ThJacobiTableFullInPartOut}--\ref{ThJacobiTablePartInPartOutInsuff} with Counterexample~\ref{CounterExNonZeroJac} (when Condition~\ref{CondOrd1Full} is violated) indicates the $(d+1)$-\/arity of the first\/-\/order `commutator' $W^{d>1}_{r=1}$ --\,for the objects on $\BBR^{d>1}$\,-- which undergoes the homotopy deformation.
An open problem is to describe the integral object for the algebra with the bracket built of $W^{d>1}_{r=1}$ and its homotopy by~$'W^{d>1}_{\underline{s>1}}$.%
\footnote{\label{FootWhatIsDeformedDimD} 
Over $d=1$, the Lie algebra $\cX^1(M^1)$ integrated to $\Diffeo(M^1)$ with associative composition~$\circ$;
the $L_\infty$-structure from \cite{ForKac,Dzhuma2002,LadaStasheff1993} integrated to an $A_\infty$-deformation of $\Diffeo(M^1)$ (note that $\circ$ is binary as $2=d+1$ for $M^1$). What is the $(d+1)$-ary analogue of the binary composition of diffeomorphisms\,?}

\subsection*{Acknowledgements}
The author thanks the organisers of the XIII International symposium on Quantum Theory and Symmetries -- QTS13 (Yerevan, Armenia, 28~July -- 1~August 2025) for a warm atmosphere during the meeting.
This work has been partially supported by the Bernoulli Institute (Groningen, NL) via project~110135.
The author thanks M.\,Kontsevich and V.\,Retakh for helpful discussions and advice.


\begin{thebibliography}{7}\normalsize

\bibitem{LeVeque1956}
\by{W.J.\ LeVeque} 
(1956) ``\book{Topics in number theory}''. Vol.~
\vol{2}.
Addison\/--\/Wesley Publ.\ 
Co., 
Reading~MA, {pp.128--134}. 

\bibitem{Schmidt1980}
\by{W.M.\ Schmidt} 
(1980) ``\book{Diophantine approximation}''.
\textit{Lect.\ Notes in Math.} vol.~\vol{785}. Springer, Berlin,
{pp.114--150}. 

\bibitem{Wolsson1989a}
\by{K.\ Wolsson} 
(1989)
Linear dependence of a function set of $m$ variables with vanishing generalized Wronskians.
\jour{Linear Algebra Appl.} \vol{117}, {73--80}\\
DOI: \href{https://doi.org/10.1016/0024-3795(89)90548-X}{https://doi.org/10.1016/0024-3795(89)90548-X}.

\bibitem{ForKac}
\by{A.V.\ Kiselev} (2007)
On associative Schlessinger\/--\/Stasheff algebras and Wronskian determinants.
\jour{J.~Math.\ Sci.
} \vol{141} 1, 1016--1030\\
DOI: \href{https://doi.org/10.1007/s10958-007-0028-2}{https://doi.org/10.1007/s10958-007-0028-2}.
(\jour{Preprint} \href{https://arxiv.org/abs/math/0410185}{\textrm{arXiv:math.RA/0410185}})

\bibitem{Dzhuma2002}
\by{A.S.\ Dzhumadil'daev} 
(2005) $n$-Lie structures generated by Wronskians.
\jour{Siberian Math.~J.} \vol{46} 
{4} {601--612}
DOI: \href{https://doi.org/10.1007/s11202-005-0061-7}{https://doi.org/10.1007/s11202-005-0061-7}.
(\textit{Preprint} \href{https://arxiv.org/abs/math/0202043}{\textrm{arXiv:math.RA/0202043}})

\bibitem{LadaStasheff1993}
\by{J.\ Stasheff, T.\ Lada} (1993)
Introduction to SH~Lie algebras for physicists.
\jour{Internat.\ J.~Theoret.\ Phys.} \vol{32} {7} {1087--1103} 
DOI: \href{https://doi.org/10.1007/BF00671791}{https://doi.org/10.1007/BF00671791}.
(\textit{Preprint} \href{https://arxiv.org/abs/hep-th/9209099}{\textrm{arXiv:hep-th/9209099}})

\bibitem{PRG25Wronsk}
\by{A.V. Kiselev} (2025)
\textrm{Wronskians as $N$-ary brackets in finite\/-\/dimensional analogues of~$\mathfrak{sl}(2)$.}
Presented at
\textit{XXIX Int.\ conf.\ on Integrable Systems \textsl{\&}\ Quantum Symmetries} 
(ISQS29, 7--11~July 2025, CVUT Prague, Czech Republic)
(\jour{Preprint} \href{https://arxiv.org/abs/2510.02145}{\textrm{arXiv:2510.02145}} [math.CO], pp.~1--8)

\end{thebibliography}
\end{document}